\documentclass[12pt]{amsart}%
\usepackage{eurosym}
\usepackage{amssymb}
\usepackage{amsmath}
\usepackage{amsfonts}
\usepackage{graphicx}%
\newtheorem{theorem}{Theorem}[section]
\newtheorem{lemma}[theorem]{Lemma}

\newtheoremstyle{notauto}{}{}{\itshape}{}{\bfseries}{.}{0.5em}{\thmnote{#3}}
\theoremstyle{notauto}

\theoremstyle{definition}
\newtheorem{definition}[theorem]{Definition}

\theoremstyle{remark}

\renewcommand{\geq}{\geqslant}
\renewcommand{\leq}{\leqslant}

\begin{document}
\title{Frobenius groups of automorphisms\\with almost fixed point free kernel}
\author{G\"{u}l\.{I}n Ercan$^{*}$}
\address{G\"{u}l\.{I}n Ercan, Department of Mathematics, Middle East Technical
University, Ankara, Turkey}
\email{ercan@metu.edu.tr}
\author{\.{I}sma\.{I}l \c{S}. G\"{u}lo\u{g}lu}
\address{\.{I}sma\.{I}l \c{S}. G\"{u}lo\u{g}lu, Department of Mathematics,
Do\u{g}u\c{s} University, Istanbul, Turkey}
\email{iguloglu@dogus.edu.tr}
\thanks{$^{*}$Corresponding author}
\thanks{This work has been supported by the Research Project T\" UB\. ITAK 114F223.}
\maketitle

\begin{abstract}
Let $FH$ be a Frobenius group with kernel $F$ and complement $H$, acting
coprimely on the finite solvable group $G$ by automorphisms. We prove that if
$C_{G}(H)$ is of Fitting length $n$ then the index of the $n$-th Fitting
subgroup $F_{n}(G)$ in $G$ is bounded in terms of $|C_{G}(F)|$ and $|F|.$ This
generalizes a result of Khukhro and Makarenko \cite{k-m} which handles the
case $n=1.$

\end{abstract}

\section{introduction}

All groups throughout this paper are finite, notation and terminology are
standard except as indicated. Let a fixed group $A$ act on the group $G$ by
automorphisms. It is well known that the structure of $A$ and the way it acts
on the group $G$ has great influence on the structure of $G$. For example, one
of the famous theorems due to Thompson says that $G$ must be nilpotent if $A$
is of prime order and acts fixed point freely on $G$. This result has played a
motivating and stimulating role in studying the structure of a group admitting
a group of automorphisms with a prescribed action. There have been a lot of
research in this direction some part of which culminated in the work of
Turull. We want to state below some of his results (\cite{Tur2}, \cite{Tur3})
in order to indicate the great development in this direction after
Thompson:\newline

\textit{Let $G$ be a solvable group and $f(G)$ denote the Fitting length of
$G$.}

$(i)$\textit{If a solvable group $A$ acts coprimely on $G$, then $f(G)\leq
f(C_{G}(A))+2\ell(A)$;}

$(ii)$\textit{If $A$ acts coprimely with regular orbits on $G$, then
$f(G)\leq\ell(A)+\ell(C_{G}(A))$, and the index $|G: F_{\ell(A)}(G)|$ is
bounded in terms of $|C_{G}(A)|$ and $|A|$, where $\ell(A)$ is the length of
the longest chain of subgroups of $A.$}\newline

In \cite{k} and \cite{k1} Khukhro studied the case where $A=FH$ is a Frobenius
group with kernel $F$ and complement $H$, and obtained very precise results
when $C_{G}(F)=1$. Namely he proved that $F_{k}(G)\cap C_{G}(H)=F_{k}%
(C_{G}(H))$ for any k, in particular $f(G)=f(C_{G}(H))$. It is worth
mentioning here the results of \cite{KhMaSh},  \cite{MaSh}, \cite{KhMaSh1}
related to the action of a Frobenius group of automorphisms with fixed point
free kernel. Although it seems to be essential to assume that $C_{G}(F)=1$ in
all these results, Khukhro and Makarenko have also considered the action of a
Frobenius group $A=FH$ with almost fixed point free kernel when $C_{G}(H)$ is
nilpotent (see \cite{k-m}, Theorem 2.1). They proved namely: \newline

\textit{If a Frobenius group $FH$ with kernel $F$ and complement $H$ acts
coprimely on the finite solvable group $G$ in such a way that $C_{G}(H)$ is
nilpotent then the index of the Fitting subgroup is bounded in terms of
$|C_{G}(F)|$ and $|F|$. }\newline

The main result of the present paper gives a new proof of the above result and extends it to the case where $C_{G}(H)$ is of arbitrary Fitting
length, that is, the group $G$ has Fitting length at most $f(C_{G}(H))$ except
for some quotient group of $G$ bounded in terms of $|C_{G}(F)|$ and $|F|$, and hereby answers a question posed by Khukhro.
Namely we obtain\newline

\textbf{Theorem} \textit{Let $A=FH$ be a Frobenius group with kernel $F$ and
complement $H$ and let $m\in\mathbb{N}$. Then there exists a function
$g: \mathbb{N}\rightarrow\mathbb{N}$ which may depend on the parameters $|F|$ and $m$, but is independent of $H$, such that for any finite solvable
group $G$ on which $A$ acts coprimely by automorphisms with $|C_{G}(F)|\leq m$
we have $|G:F_{n}(G)|\leq g(n)$ where $n=f(C_{G}(H)).$}\newline

\section{Preliminaries}

\begin{lemma}
[\cite{Tur3}, Lemma 2.4]Let $G$ be a solvable group and suppose that $N_{i},
i=1,\ldots,r$ be normal subgroups of $G$ such that $\bigcap_{i=1}^{r}
N_{i}=1.$ Let $n\in\mathbb{N}$. Then
\[
|G:F_{n}(G)|\leq\prod_{i=1}^{r}|G/N_{i}/F_{n}(G/N_{i})|.
\]

\end{lemma}

\begin{lemma}
[\cite{k}, Lemma 1.3]Let $FH$ ve a Frobenius group with kenel $F$ and
complement $H$. Suppose that $V$ is a vector space over an arbitrary field on
which $FH$ acts by linear transformations. If $[V,F]\ne0$ then $C_{V}(H)\ne0.$
\end{lemma}

\begin{lemma}
[\cite{EG}, Proposition 4.1]Let $FH$ be a Frobenius group with kernel $F$ and
complement $H$ acting on a $q$-group $Q$ for some prime $q$ coprime to the
order of $FH$. Let $V$ be a $kQFH$-module where $k$ is a field of
characteristic not dividing $|QFH| $. If $C_{V}(F)=1$ then
\[
Ker(C_{Q}(H) \,{on}\, C_{V}(H))=Ker(C_{Q}(H) \,{on}\, V).
\]

\end{lemma}

\begin{definition}
Let $G$ be a group. We call a sequence $P_{i}=C_{i}/D_{i}, i=1,\ldots,\ell$ of
sections of $G$ an $\mathcal{F}$-chain in $G$ if the following are
satisfied:\newline

$(a)\,\, P_{i}$ is a $p_{i}$-group for a single prime $p_{i}$ for
$i=1,\ldots,\ell$;\newline

$(b)\,\, p_{i}\ne p_{i+1}$ for $i=1,\ldots,\ell-1$;\newline

$(c)\, D_{\ell}=1$ and $D_{i}\leq C_{C_{i}}(P_{i+1})$ for
$i=1,\ldots,\ell-1$;\newline

$(d)\, [P_{i},P_{i+1}]=P_{i+1}$ for
$i=1,\ldots,\ell-1$.
\end{definition}

\begin{lemma}
In a solvable group $G$, the Fitting length $f(G)$ of $G$ is equal to the
maximum of the set $\{\ell: \ell\,\, \text{is the length of an }%
\mathcal{F}\text{-chain in}\,G \}$.
\end{lemma}

\begin{proof}
Let $m=m(G)=\max\{\ell: \ell\,\, \text{is the length of an }\mathcal{F}%
\text{-chain in}\,G \}$ and $f=f(G).$

It is well known that there is an irreducible tower in $G$, say $\hat{P_{i}},
i=1,\ldots,f$, in the sense of \cite{Tur2}. We observe now that the
corresponding sequence $P_{i}, i=1,\ldots,f,$ where $P_{f}=\hat{P_{f}}$ and
$P_{i}=\hat{P_{i}}/C_{\hat{P_{i}}}(P_{i+1})$ for $i=1,\ldots,f-1$, forms an $\mathcal{F}$-chain of length $f$ in $G$ in the sense of Definition 2.4. Thus we have $f\leq
m$.

We shall prove the reversed inequality by induction on $|G|.$ Let $P_{i}=C_{i}%
/D_{i}, i=1,\ldots,m,$ be an $\mathcal{F}$-chain in $G$ of length $m.$ Now
$P_{m}$ is a $p_{m}$-group. Suppose first that there exists a prime $p$
different from $p_{m}$ such that $O_{p}(G)\ne 1$. Set $\bar{G}=G/O_{p}(G).$
Since $\bar{C}_{m}\ne\bar{D}_{m}$, the chain $P_{i}=C_{i}/D_{i},
i=1,\ldots,m,$ is mapped to an $\mathcal{F}$-chain in $\bar{G}$ of length $m$
and so $m(\bar{G})\geq m$. Thus we have
\[
m\leq m(\bar{G})\leq f(\bar{G})\leq f(G)
\]
by induction. Then we may assume that $F(G)=O_{p_{m}}(G).$ Set now $\tilde
{G}=G/F(G).$ Notice that $\tilde{C}_{m-1}\ne\tilde{D}_{m-1}$ and hence the
chain $P_{i}=C_{i}/D_{i}, i=1,\ldots,m-1,$ is mapped to an $\mathcal{F}$-chain
in $\tilde{G}$ of length $m-1.$ It follows then by induction that
\[
m-1\leq m(\tilde{G})\leq f(\tilde{G})=f(G)-1
\]
which completes the proof.
\end{proof}

Finally we want to state a special case of Hartley-Isaacs theorem [\cite{H}, Theorem B].

\begin{lemma} For any arbirary group $F$ there is a number $\delta(F)$ depending only on $F$ with the following property: Let $F$ act coprimely on the solvable group $G,$ and let $k$ be any field of characteristic not dividing $|F|.$ Then for any completely reducible $kGF$-module $V$,  we have $dim_{k}V\leq \delta(F)dim_{k}C_V(F).$

\end{lemma}
\section{Proof of Theorem}

Let $A=FH$ be a fixed Frobenius group with kernel $F$ and complement $H$ and
$m$ be a fixed positive integer and $\mathcal{G=G}_{A,m}$ be the set of all
finite solvable groups $G$ on which $A$ acts coprimely by automorphisms with
$|C_{G}(F)|\leq m$. Clearly, we have $\mathcal{G}=\bigcup_{n\in\mathbb{N}%
}\mathcal{G}_{n}$ where
\[
\mathcal{G}_{n}=\{G\in\mathcal{G}:f(C_{G}(H))=n\}.
\]
Our theorem needs to prove that the subset $\{|G:F_{n}%
(G)|:G\in\mathcal{G}_{n}\}$ of $%
%TCIMACRO{\U{2115} }%
%BeginExpansion
\mathbb{N}
%EndExpansion
$ has a maximum for any $n\in\mathbb{N}$. If this is known one
can define a function $h=h_{A,m}:%
%TCIMACRO{\U{2115} }%
%BeginExpansion
\mathbb{N}
%EndExpansion
\rightarrow%
%TCIMACRO{\U{2115} }%
%BeginExpansion
\mathbb{N}
%EndExpansion
$ such that
\[
h(n)=\max\{|G:F_{n}(G)|:G\in\mathcal{G}_{n}\}
\]
for any $n.$ Of course this function, if it exists, may depend on the
parameters $m,F$ and $H$. Our proof will be given by a recursive construction
of the function $h$ and will show that $h$ may depend on $F$, but is actually
independent of $H.$ We next define the function $g: \mathbb{N}\rightarrow\mathbb{N}$ by $$g(n)=\max\{h_{A,m}(n):A\text{ is a Frobenius
group with kernel of order} \left\vert F\right\vert \}.$$ Notice that this maximum exists as the set of Frobenius groups with kernel of a given order is finite. It is also clear that $g(n)$ depends only on $\left\vert F\right\vert $ and $m$, and satisfies $|G:F_{n}(G)|\leq g(n)$ for any group $G$ on which $A$ acts coprimely by automorphisms with
$|C_{G}(F)|\leq m$ and $f(C_{G}(H))=n.$

We proceed now to prove by induction on $n$ that the set $\{|G:F_{n}(G)|:G\in\mathcal{G}_{n}\}$ has a maximum $h(n)$. Suppose first that $n=0$ and let
$G\in\mathcal{G}_{0} $. Now $f(C_{G}(H))=0$, that is $C_{G}(H)=1$. This can
happen only when $[G,F]=1$ by Lemma 2.2. Therefore $|G|\leq m$ and so $h(0)$
exists. We assume that for any fixed $n\geq1$ and any $k<n$ we have $\{
|G:F_{k}(G)| : G\in\mathcal{G}_{k}\}$ has a maximum $h(k)$. Set $d _{n}=\max\{g_{k}: k=0,\ldots,n-1\}$.

Let now $G$ be a fixed, but arbitrary element of $\mathcal{G}_{n}$. Notice
that if $|F_{n+1}(G):F_{n}(G)|$ is bounded by a number which depends only on
$m$ then so is $|G:F_{n+1}(G)|$. Therefore there is no loss in assuming that
$G=F_{n+1}(G)$. As $|G/[G,F]|\leq m,$ without loss of generality we may assume
that $[G,F]=G$. Since $F(G)/\Phi(G)=F(G/\Phi(G))$, we may also assume that
$\Phi(G)=\Phi(F(G))=1$. It is well known that $F(GA)/\Phi(GA)$ is a direct sum
of irreducible $GA$-submodules over possibly different prime fields . Since
$F(G)\Phi(GA)/\Phi(GA)$ is a $GA$-submodule of $F(GA)/\Phi(GA)$ and is
isomorphic to $F(G)/\Phi(G)$ as a $GA$-group we see that $F(G)$ is a direct
sum of irreducible $GA$-modules $M_{1},\ldots,M_{k}$. On the other hand, by
Clifford's theorem each $M_{i}$ is a direct sum of irreducible $G$-modules
$M_{ij}, j=1,\ldots,s_{i}$. By A.13.8(b) in \cite{Do}, we see that
$F(G)=\bigcap_{i=1}^{k}\bigcap_{j=1}^{s_{i}}C_{G}(M_{ij})=\bigcap_{i=1}%
^{k}C_{G}(M_{i})$.

Without loss of generality we may suppose that $C_{M_{i}}(F)\ne0$ for $i=1,\ldots,s$ and $C_{M_{i}}(F)=0 $ for
$i=s+1,\ldots,k.$ Set $\bar{G}=G/F(G)$ and $\bar{X}%
=\bigcap_{i=1}^{s}C_{\bar{G}}(M_{i})$ and $\bar{Y}=\bigcap_{i=s+1}^{k}%
C_{\bar{G}}(M_{i}).$ Then $\bar{X}$ and $\bar{Y}$ are subgroups of $\bar{G}$
with $\bar{X}\cap\bar{Y}=1.$ Clearly, by Lemma 2.1
\[
|G:F_{n}(G)|=|\bar{G}/F_{n-1}(\bar{G})|\leq|\bar{G}/\bar{X}/F_{n-1}(\bar
{G}/\bar{X})||\bar{G}/\bar{Y}/F_{n-1}(\bar{G}/\bar{Y})|.
\]

On the other hand
\[
\prod_{i=1}^{s}|C_{M_{i}}(F)|\leq|C_{G}(F)|=m
\]
and so
\[
2^{s}\leq\prod_{i=1}^{s} |C_{M_{i}}(F)|\leq m.
\]
This gives $s\leq\log_{2}m$ as $|C_{M_{i}}(F)|\geq2$ for any $M_{i}$ where
$i=1,\ldots s$. Therefore we have%

\[
|\bar{G}/\bar{X}|\leq\prod_{i=1}^{s}|G/C_{G}(M_{i})| \leq\prod_{i=1}^{s}|Aut
M_{i}| \leq\prod_{i=1}^{s}|M_{i}|!.
\]

Pick $M_{i}$ for $i\in\{1,\ldots,s\}$ and consider its decomposition into
$GF$-homogeneous components. Note that $H$ acts transitively on the set of
$GF$-homogeneous components, and hence $F$ fixes a point in each component.
Then by Lemma 2.6, there is a number $\delta(F)$ depending only on $F$ such that
\[
|M_{i}|\leq|C_{M_{i}}(F)|^{\delta(F)}\leq m^{\delta(F)}.
\]
Thus we have%

\[
|\bar{G}/\bar{X}|\leq({m^{\delta(F)}!})^{\log_{2}m}.
\]

Note that this bound depends on $F$ and $m$, but is completely independent of $H.$ Clearly if $s=k$ then $\bar{Y}=\bar{G}$ and the result follows from the
above equality. Therefore we may assume that $s\ne k$. Set $G_1=G/\bigcap_{i=s+1}^{k}C_{G}(M_{i})$. As $C_{\bar{G}}(M_{i})=C_{G}(M_{i})/F(G)$
we have $G_1\cong \bar{G}/\bar{Y}$ and so
it suffices to bound $|G_{1}:F_{n-1}(G_{1})|$ suitably. If $f(C_{G_{1}%
}(H))\leq n-1$ then it would follow by induction assumption that $|G_{1}:F_{n-1}%
(G_{1})|\leq d _{n} $. Thus we have $f(C_{G_{1}}(H))=n=f(G_{1})$ as we already
have $f(G_{1})\leq n.$

Now by Lemma 2.5, there exists a chain of sections $P_{i}=C_{i}/D_{i},
i=1,\ldots,n$, of $C_{G_{1}}(H)$ satisfying the conditions $(a)-(d)$
of Definition 2.4. Clearly, $P_{i}=C_{i}/D_{i},\, i=1,\ldots,n$, is also a chain
of sections of $G_{1}$. Furthermore, we have
$P_{n}=C_{n}\leq F(G_{1})$ because otherwise the chain $P_{i}, i=1,\ldots,n$, is mapped to a chain of the same length in $G_1/F(G_1),$ which is impossible. As a $p_{n}$-subgroup, $C_{n}$ is contained in the Sylow $p_{n}$-subgroup, say $Q$, of $F(G_{1})$.
Clearly $Q$ is $A$-invariant. Notice that there exists $M=M_{i} $ for some
$i\in\{s+1,\ldots,k\}$ on which $C_{n}$ acts nontrivially because otherwise
$C_{n}$ is the trivial subgroup of $G_{1}$. We consider now the action of $A$
on the group $MQ$ and apply Lemma 2.3. It follows that
\[
Ker(C_{Q}(H)\,\, on\,\, C_{M}(H))=Ker(C_{Q}(H)\,\, on\,\, M).
\]
Due to coprimeness, $[C_{M}(H),P_{n},P_{n}]=[C_{M}(H),P_{n}]$. We set now
$P_{n+1}=[C_{M}(H),P_{n}]$. Clearly the sequence $P_{1},\ldots, P_{n},
P_{n+1}$ forms a chain of sections of $C_{G}(H)$ satisfying the conditions
$(a)-(d)$ of Definition 2.4 as $C_{G_1}(H)$ is the image of $C_G(H)$ in $G_1$ by the coprimeness condition $(|G|,|H|)=1$. This forces by Lemma 2.5 that
$f(C_{G}(H))=n+1$, which is a contradiction completing the proof. $\Box$

\end{document}